\def\|{|\;}

\documentclass[11 pt]{amsart}
\usepackage{amssymb}
\usepackage{amsmath}
\usepackage{amsthm}
\usepackage{amscd}

\theoremstyle{plain}

\newtheorem{theorem}{Theorem}[section]
\newtheorem{proposition}[theorem]{Proposition}
\newtheorem{lemma}[theorem]{Lemma}
\newtheorem{corollary}[theorem]{Corollary}

\theoremstyle{definition}
\newtheorem{remark}[theorem]{Remark}

\numberwithin{equation}{section}
\begin{document}


\title[The Hankel transform]{The classical Hankel transform
in the Kirillov model}
\author{Ehud Moshe Baruch}
\address{Department of Mathematics\\
         Technion\\
         Haifa , 32000\\
         Israel}
\email{embaruch@math.technion.ac.il}
\thanks{}

\date{October 2009}

\begin{abstract}
We give a new and simple proof of the Hankel inversion formula for the classical 
Hankel transform of index $\nu$ which holds for $\text{Re}(\nu)>-1$. Using the proof 
of this formula we obtain the full description of the Kirillov model for discrete series 
representations of  $SL(2,\mathbb{R})$ and $GL(2,\mathbb{R})$. 
\end{abstract}
\maketitle

\section{Introduction}
It has been noticed by \cite{cps1} (See also \cite{ggps},\cite{vil},\cite{mot}, \cite{bm3}) that 
the action of the Weyl element in the Kirillov model
of an irreducible unitary representation of $GL(2,\mathbb{R})$ is given by a
certain integral transform and that in the case of the discrete series this is
a classical Hankel transform of integer order. 
It was proved in \cite{dur} that the
classical Hankel transform of real order $\nu$ with $\nu>-1$  is an isomorphism of 
order two of a ceratin ``Schwartz'' space.  In this paper we give a simple and 
elementary proof of this fact using ``representation theoretic'' ideas and extend the 
result to complex $\nu$ with $\text{Re}(\nu)> -1$. Using these methods for the Hankel
transform of integral order we determine the smooth space of the discrete series
in the Kirillov model, thus giving for the first time an explicit model which is suitable
for various applications. In a future paper we will describe the Kirillov model for
the principal series and the complementary series and give an inversion formula
for a generalized Hankel transform.  

The inversion formula for the Hankel transform which states that the Hankel transform
is self reciprocal was studied by many authors starting from Hankel in \cite{han}. This
formula is classicaly stated as follows (\cite{wat} p.453).  Let  $J_{\nu}(z)$ be the 
classical J-Bessel 
function. Let $f$ be a complex valued function defined on the positive real line.
Then under certain assumptions on $f$ and $\nu$ (See \cite{wat} or Theorem~\ref{T:hankinv}) 
we have
\begin{equation} \label{E:Hankelinv}
f(z)=\int_{0}^{\infty}\int_{0}^{\infty}f(x)J_{\nu}(xy)J_{\nu}(yz)xydxdy
\end{equation}
In more modern notation we define the Hankel transform of order $\nu$ of $f$ to be 
$$
\mathit{h}_{\nu}(f)(y)
=\int_{0}^{\infty}f(x)J_{\nu}(xy)xdx
$$
Then under certain assumptions on $f$ and $\nu$ the Hankel transform is self 
reciprocal, that is, $\mathit{h}_{\nu}^{2}=Id$. A general discussion of the history
of this result and various attempts at a proof can be found in (\cite{wat}, p. 454). 
It is mentioned in \cite{wat}
that Hankel was the first to give the formula in 1869 and that Weyl (\cite{weyl} p.324)  
was the first to give a complete proof when $f$ is in a certain
space of twice differential functions and $\nu>-1$. Watson gives a complete
proof when $\nu$ is real, $\nu> -1/2$ and $\phi\in L^{1}((0,\infty),\sqrt{x}dx)$ and is 
of bounded variation in an interval around the point $z$ in 
equation~\ref{E:Hankelinv}. MacRobert \cite{mac},  proves the result for 
$\text{Re}(\nu)> -1$ under the additional assumption that $\phi$ is analytic. In more
recent results it was proved by \cite{zem} that the Hankel transform preserves a certain
``Schwartz'' space. Duran \cite{dur} gives a very elegant proof of the inversion 
formula on the ``Schwartz'' space when $\nu>-1$. In this paper we
will extend the inversion formula on a ``Schwartz'' space to complex $\nu$ such that 
$\text{Re}(\nu)> -1$. Our method 
is to ``move'' the Hankel
transform from the Schwartz space on which it acts to a different space using 
Fourier transform which we think of as an ``intertwining operator'' . The main part
of the proof is to compute the effect of this ``intertwining operator'' on the Hankel
transform. The operator is built in such a way that the Hankel transform is replaced
with an operator which sends a function $\phi$ on the real line to the function
$|x|^{-\nu-1}e^{\text{sgn}(x)\pi i(\nu+1)/2}\phi(-1/x)$. It is trivial to see that this operator
is self reciprocal and from this follows the same result for the Hankel transform.

In the case where $\nu$ is a positive integer, the operator above is giving the action
of the Weyl element of $SL(2,\mathbb{R})$ or $GL(2,\mathbb{R})$ on the
discrete series representations.  The operator which moves us from the action
above to the Hankel transform is an honest intertwining operator for a discrete
series representations which moves us from a model for the induced
space to the Kirillov model. Using the results above we can determine completely
the ``smooth'' space of this Kirillov model and give an explicit action in this model.

Our paper is divided as follows. In section 2 we describe the Schwartz space
for the general Hankel transform of order $\nu$ with $\text{Re}(\nu)>-1$ and
prove that the Hankel transform preserves this space. In section 3 we define
our `` intertwining operator'' and compute its composition with the Hankel transform.
Using that we prove the inversion formula for the Hankel transform. In section 4
we turn our attention to integer order and to the discrete series representations.
We compute the full image of the induced space into the ``Kirillov space''.
In section 5 we show that two subspaces in the full ``Kirillov''  space
are invariant. These subspaces which are basically the Schwartz spaces mentioned
above are irreducible under the action of $SL(2,\mathbb{R})$ 
and give a realization of the discrete series representations. Their sum gives
the classical Kirillov model for discrete series representations of $GL(2,\mathbb{R})$.

It is our purpose to make the presentation elementary and suitable for non 
representation theory specialists. We will use results from representation 
theory in sections 5 and 6 but the rest of the material will be self 
contained.

\section{A Schwartz space for the classical Hankel transform}
In this section we describe the Schwartz space (\cite{and},\cite{dur},\cite{zem}) for the Hankel 
transform and show that this space is invariant under the Hankel transform. While the proofs are
elementary and the results may follow from the references above, we view this section as crucial
for our results and we provide the proofs for completeness.   

Assume that $\text{Re}(\nu)>-1$ and  $x>0$ and  let  $J_{\nu}(x)$ be the 
classical J-Bessel function defined by
$$
J_{\nu}(x)=\sum_{k=0}^{\infty}\frac{(-1)^k(x/2)^{\nu+2k}}{\Gamma(k+1)\Gamma(k+\nu+1)}
$$
Let $f$ be a complex valued function on $[0,\infty)$. Define the Hankel transform of order $\nu$ by
$$
\mathcal{H}_{\nu}(f)(y)=
\int_{0}^{\infty}f(x)\sqrt{xy}J_{\nu}(2\sqrt{xy})\frac{dx}{x}.
$$
It is possible by a simple change of variable to go from this Hankel transform to the
``classical'' Hankel transform used in \cite{wat}. Let $S([0,\infty))$ be the Schwartz space
of functions on $[0,\infty)$. That is, $f:[0,\infty)\to \mathbb{C}$ is in $S([0,\infty))$  if
$f$  is smooth on  $[0,\infty)$ and $f$ and all its derivatives are rapidly decreasing
at $\infty$. Let 
$$
S_{\nu}([0,\infty))=\{f:[0,\infty)\to \mathbb{C} | f(x)=x^{1/2+\nu/2}f_{1}(x) 
\;\text{and} \;f_{1}\in S([0,\infty))\}
$$
It follows from \cite{dur} that when $\nu$ is real, $\nu> -1$ then $H_{\nu}$ is a linear
isomorphism of  $S_{\nu}([0,\infty))$ satisfying $H_{\nu}^2=Id$.   
Moreover, $H_{\nu}$ is an $L_{2}([0,\infty),dx/x)$ isometry. We will give a simple proof
that  $H_{\nu}$ is a linear isomorphism of  $S_{\nu}([0,\infty))$ satisfying $H_{\nu}^2=Id$ when
$\text{Re}(\nu)>-1$. We start by showing that $H_{\nu}$ preserves $S_{\nu}([0,\infty))$.
 
Let $D$ be the differential operator
$$
D=x\frac{d}{dx}.
$$

\begin{proposition}

(a) $D$ maps $S_{\nu}([0,\infty))$ into $S_{\nu}([0,\infty))$

(b) $\mathcal{H}_{\nu}(f)(y)$ is rapidly decreasing in $y$.

(c) $\mathcal{H}_{\nu}(Df)(y)=-D(\mathcal{H}_{\nu}(f))(y)$.

(d) $y^{-1/2-\nu/2}\mathcal{H}_{\nu}(f)(y)$ is smooth on $[0,\infty)$.
\end{proposition}
\begin{proof}
(a) is immediate. (b), (c) and (d) are standard using differentation under the
integral and integration by parts. We include the proof for the sake of completeness.

We will need the following formulas:

$$\frac{d}{dx}\left (y^{-1/2}x^{(\nu+1)/2}J_{\nu +1}(2x^{1/2}y^{1/2})\right)
=x^{\nu/2}J_{\nu}(2x^{1/2}y^{1/2})
$$
and
$$
\frac{d}{dy}\left ( y^{-\nu/2}J_{\nu +1}(2x^{1/2}y^{1/2})\right)
=-x^{-1/2}y^{-(\nu+1)/2}J_{\nu}(2x^{1/2}y^{1/2})
$$
These formulas follow from the standard differentiation formulas of the Bessel function
(\cite{leb}, (5.3.5)).

We now prove (b). We shall use the first formula to do a repeated integration by
parts on the integral defining $H_{\nu}$.  Let $f(x)=x^{\nu/2+1/2}f_{1}(x)$ where $f_{1}\in S([0,\infty))$.
Then
\begin{equation} 
\begin{split}
\mathcal{H}_{\nu}(f)(y)&=
y^{1/2}\int_{0}^{\infty}f_{1}(x)\left( x^{\nu /2}J_{\nu}(2\sqrt{xy})\right)dx\\
&=\int_{0}^{\infty}\frac{d}{dx}\left( f_{1}(x) \right)
x^{(\nu+1) /2}J_{\nu+1}(2\sqrt{xy})dx
\end{split} 
\end{equation}
Repeating this process $n$ times will give
$$
\mathcal{H}_{\nu}(f)(y)=
y^{1-n}\int_{0}^{\infty}\frac{d^{n}}{(dx)^{n}}\left( f_{1}(x) \right)
(xy)^{(\nu+n) /2}J_{\nu+n}(2\sqrt{xy})dx
$$
Since the function $z^{\mu}J_{\mu}(z)$ is bounded for $Re(\mu)>0$ on $[0,\infty)$ and the
derivatives of $f_{1}$ are rapidly decreasing it follows
that the integral above is bounded hence $|\mathcal{H}_{\nu}(f)(y)|<<y^{1-n}$ and
we have that (b) holds.

For (c) we write
\begin{equation}
\begin{split}
-D(\mathcal{H}_{\nu}(f))(y)&=
-\int_{0}^{\infty}f(x)x^{-1/2}y\frac{d}{dy}
\left(y^{-\nu/2+1/2} y^{\nu /2}J_{\nu}(2\sqrt{xy})\right)dx \\
&=(\nu/2-1/2)\mathcal{H}_{\nu}(f)(y)-\int_{0}^{\infty}f(x)yJ_{\nu-1}(2\sqrt{xy})dx\\
&=(\nu/2-1/2)\mathcal{H}_{\nu}(f)(y)
-y\int_{0}^{\infty}\left(f(x)x^{\nu/2-1/2}\right) x^{-(\nu-1)/2}J_{\nu-1}(2\sqrt{xy})dx\\
&=(\nu/2-1/2)\mathcal{H}_{\nu}(f)(y)
+y^{1/2}\int_{0}^{\infty}\frac{d}{dx}\left(f(x)x^{-\nu/2+1/2}\right) 
x^{\nu/2}J_{\nu}(2\sqrt{xy})dx\\
&=\int_{0}^{\infty} (xf'(x))y^{1/2}J_{\nu}(2\sqrt{xy})\frac{dx}{x}\\
&=\mathcal{H}_{\nu}(Df)(y)
\end{split}
\end{equation}
To prove (d) we will differentiate under the integral. It is easy to see
that the integral converges absolutely and uniformely in $y$ hence we are
allowed to differentiate: 
\begin{equation}
\begin{split}
\frac{d}{dy}\left( y^{-1/2-\nu/2}\mathcal{H}_{\nu}(f))(y) \right)&=
\int_{0}^{\infty}f(x)x^{-1/2}\frac{d}{dy}\left( y^{-\nu /2}J_{\nu}(2\sqrt{xy})\right)dx \\
&=-\int_{0}^{\infty}f(x)y^{-(\nu+1)/2} J_{\nu+1}(2\sqrt{xy})dx\\
&=-\int_{0}^{\infty}f(x)x^{(\nu+1)/2}(xy)^{-(\nu+1)/2} J_{\nu+1}(2\sqrt{xy})dx
\end{split}
\end{equation}
Since $z^{-(\nu+1)/2}(J_{\nu}(z))$ is analytic and bounded on $[0,\infty)$ it follows
that  $y^{-1/2-\nu/2}\mathcal{H}_{\nu}(f))(y)$ is differentiable on $[0,\infty)$.
Continuing in the same way we can show that 
$$
\frac{d^{n}}{(dy)^{n}}\left(y^{-1/2-\nu/2}\mathcal{H}_{\nu}(f))(y) \right)=
(-1)^{n}\int_{0}^{\infty}f(x)x^{(\nu+n)/2}(xy)^{-(\nu+n)/2} J_{\nu+n}(2\sqrt{xy})dx
$$
hence $y^{-1/2-\nu/2}\mathcal{H}_{\nu}(f))(y)$  is smooth on $[0,\infty)$.
\end{proof}
\begin{corollary} \label{C:stab}
Assume that $\text{Re}(\nu)> -1$. Then $\mathcal{H}_{\nu}$ maps $S_{\nu}([0,\infty))$
into itself.
\end{corollary}
\begin{proof}
Let $f\in S_{\nu}([0,\infty))$. It follows from (b) (or (d)) that $\mathcal{H}_{\nu}(f)(y)$ is smooth
on $(0,\infty)$ and that $D^{n}( \mathcal{H}_{\nu}(f))(y)$ is rapidly decreasing for
$n=0,1,2,\ldots$ It follows immediately that $\mathcal{H}_{\nu}(f))(y)$ and all its
derivatives are rapidly decreasing. By (d) 
$g(y)=y^{-1/2-\nu/2}\mathcal{H}_{\nu}(f))(y)$  is smooth on $[0,\infty)$ and it follows that
all its derivatives are rapidly decreasing hence we can write
$ \mathcal{H}_{\nu}(f)(y)=y^{1/2+\nu/2}g(y)$ for $g\in S([0,\infty))$.
\end{proof}

\section{The inversion formula for the Hankel transform}

We now turn to prove the self reciprocity of the Hankel transform.
The crucial idea is to define
an operator (``intertwining operator'') $T_{\nu}$ from the Schwartz space above to 
another
space of smooth functions (but not Schwartz in the usual sense) which we will
call $I_{\nu}$ and to compute the composition of the Hankel transform with this operator.
We will find an operator $\mathcal{W}_{\nu}$ on  $I_{\nu}$ so that the  following diagram commutes:
$$
\begin{CD}
S_{\nu}@>\mathcal{H}_{\nu}>>S_{\nu}\\
@VT_{\nu}VV @VVT_{\nu}V\\
I_{\nu}@>\mathcal{W}_{\nu}>>I_{\nu}\\ 
\end{CD}
$$
We now define the above operators. Let $f\in S_{\nu}([0,\infty))$. Extend $f$ to
$(-\infty,\infty)$ by setting it to be zero on the negative reals. We will denote this extension again by $f$.  Assume $\text{Re}(\nu)>-1$. Let $T_{\nu}(f)=(x^{-1/2+\nu/2}f)^{\lor}$ where
$\check{f}$ is the inverse Fourier transform. That is, for $z\in \mathbb{R}$ we let
\begin{equation} \label{E:Tmu}
T_{\nu}(f)(z)=(2\pi)^{-1/2}\int_{0}^{\infty}x^{-1/2+\nu/2}f(x)e^{i xz}dx
\end{equation}
Since $f$ is rapidly decreasing at $\infty$ and of order $x^{1/2+\nu/2}$ at $x=0$ it follows
that the integral is absolutely convergent. From the Plancherel theorem and the fact that
$f$ is continuous on $(0,\infty)$ it follows that $T_{\nu}$ is one to one.
We will show that $T_{\nu}$ maps $S_{\nu}([0,\infty))$ into a space of function $I_{\nu}$ which
will be defined below. First we define the operator $\mathcal{W}_{\nu}$: Let
$\phi:\mathbb{R}^{*}\to \mathbb{C}$. Define
$$
\mathcal{W}_{\nu}(\phi)(x)=|x|^{-\nu-1}e^{\text{sgn}(x)\pi i(\nu+1)/2}\phi(-1/x).
$$
Our main theorem of this section is the following:
\begin{theorem} \label{T:main}
Let $f\in S_{\nu}([0,\infty))$ then $T_{\nu}\circ \mathcal{H}_{\nu}(f)=
\mathcal{W}_{\nu}\circ T_{\nu}(f)$.
\end{theorem}
Notice that this theorem is precisely the statement that the above diagram 
is commutative. We regard it as our main theorem since it allows us to move
from the complicated Hankel transform $\mathcal{H}_{\nu}$ to the simple operator
$\mathcal{W}_{\nu}$.
\begin{proof}
The proof is based on the Weber integral (\cite{leb}, p.132):
\begin{equation} \label{GR1}
\int_{0}^{\infty}u^{\nu+1}e^{-\alpha u^{2}}J_{\nu}(\beta u)du=\frac{\beta^{\mu}}{(2\alpha)^{\mu+1}}e^{-\frac{\beta^2}{4\alpha}}
\end{equation}
where $\text{Re}(\mu)>-1$, $\text{Re}(\alpha)>0$ and  
$\text{Re}(\beta)>0$.
Assume that $\text{Re}(\nu)>-1$.  It follows from the dominated convergence theorem
that if $f\in S_{\nu}([0,\infty))$ then 
$$
T_{\nu}(f)(z)=\lim_{\epsilon \to 0^{+}}(2\pi)^{-1/2}\int_{0}^{\infty}y^{-1/2+\nu/2}f(y)e^{i yz}e^{-\epsilon y}dy
$$
Hence  
\begin{equation}
\begin{split}
&T_{\nu}\circ \mathcal{H}_{\nu}(f)(z)=\\
&=\lim_{\epsilon \to 0^{+}}(2\pi)^{-1/2}\int_{0}^{\infty}y^{-1/2+\nu/2}e^{i yz}e^{-\epsilon y}\int_{0}^{\infty}f(x)\sqrt{xy}
J_{\nu}(2\sqrt{xy})\frac{dx}{x}\, dy\\
&= \lim_{\epsilon \to 0^{+}}(2\pi)^{-1/2}\int_{0}^{\infty}f(x)x^{-1/2}
\int_{0}^{\infty}y^{\nu/2}e^{-(-iz+\epsilon)y}J_{\nu}(2\sqrt{xy})dydx\\
&= \lim_{\epsilon \to 0^{+}}2(2\pi)^{-1/2}\int_{0}^{\infty}f(x)x^{-1/2}
\left(\int_{0}^{\infty}u^{\nu+1}e^{-(-iz+\epsilon)u^2}J_{\nu}(2\sqrt{x}u)du\right)dx\\
&\overset{(\ref{GR1})}{=}\lim_{\epsilon \to 0^{+}}(2\pi)^{-1/2}\int_{0}^{\infty}f(x)x^{-1/2+\nu/2}(-iz+\epsilon)^{-\nu-1}e^{-\frac{x}{\epsilon -iz}}dx\\
&=(2\pi)^{-1/2}|z|^{-\nu-1}e^{\text{sgn}(z)\pi i(\nu+1)/2}\int_{0}^{\infty}f(x)x^{-1/2+\nu/2}e^{\frac{ix}{-z}}dx\\
&=|z|^{-\nu-1}e^{\text{sgn}(z)\pi i(\nu+1)/2}T_{\nu}(f)(-1/z).\\
\end{split}
\end{equation}

\end{proof}
As a Corollary we obtain the inversion formula of the Hankel transform on the
Schwartz space. For the inversion formula on a bigger space see \cite{wat}.
\begin{theorem} \label{T:hankinv}
Assume $\text{Re}(\nu)>-1$ and $f\in S_{\nu}([0,\infty))$. Then
$$
\mathcal{H}_{\nu}\circ \mathcal{H}_{\nu}(f)=f
$$
\end{theorem}
\begin{proof}
This follows immediately from Theorem~\ref{T:main} and the self reciprocity of 
$\mathcal{W}_{\nu}$,
$$
\mathcal{W}_{\nu}\circ \mathcal{W}_{\nu}=Id
$$
which is easy to check. The argument is the following. Let $f\in S_{\nu}([0,\infty))$. Then 
$$
T_{\nu}\circ \mathcal{H}_{\nu}\circ \mathcal{H}_{\nu}(f)=
\mathcal{W}_{\nu}\circ T_{\nu}\circ \mathcal{H}_{\nu}(f)=
\mathcal{W}_{\nu}\circ \mathcal{W}_{\nu}\circ T_{\nu}(f)=T_{\nu}(f).
$$
Since $T_{\nu}$ is one to one, it follows that $\mathcal{H}_{\nu}\circ \mathcal{H}_{\nu}(f)=f$.
\end{proof}
We can also define the space $I_{\nu}$ (which we think of as an induced representation
space).  
\begin{equation} \label{E:Inu}
I_{\nu}=\{\phi :\mathbb{R}\to \mathbb{C}\;| \;
\text{$\phi $ is smooth on $\mathbb{R}$ and 
$\mathcal{W}_{\nu}(\phi)$ is smooth on $\mathbb{R}$} \}  . 
\end{equation}
\begin{proposition} \label{P:map}
$T_{\nu}$ maps $S_{\nu}([0,\infty))$ into $I_{\nu}$.
\end{proposition}
\begin{proof}
 Let $f\in S_{\nu}([0,\infty))$. We will show that $\phi=T_{\nu}(f)\in I_{\nu}$. Let
$g(x)= x^{-1/2+\nu/2}f(x)$. Then 
$T_{\nu}(f)=\check{g}$. Since $x^{n}g(x)$ is absolutely integrable for
every positive integer $n$, it follows from standard Fourier analysis that
$\phi=\check{g}$ is smooth. Now apply $T_{\nu}$ to $\mathcal{H}_{\nu}(f)$. Since
$\mathcal{H}_{\nu}(f)\in S_{\nu}([0,\infty))$ it follows from the same argument that
$T_{\nu}(\mathcal{H}_{\nu}(f))$ is smooth. But from Theorem~\ref{T:main} this is
precisely $\mathcal{W}_{\nu}(T_{\nu}(f))=\mathcal{W}_{\nu}(\phi)$.
\end{proof}
\subsection{$L^2$ Isoemtry of the real Hankel transform}
The $L^{2}$ isometry of the Hankel transform is well known and follows from the Plancherel 
formula for the Hankel transform. (See for example \cite{and})

:  For $f_{1},f_{2}\in S_{\nu}([0,\infty))$ we let
$$
<f_{1},f_{2}> = \int_{0}^{\infty}f_{1}(x)\overline{f_{2}(x)}\frac{dx}{x}
$$
where this integral converges absolutely when $\text{Re}(\nu)>-1$.
\begin{proposition}
Assume $\nu$ is real and $\nu>-1$. Let $f_{1},f_{2}\in S_{\nu}([0,\infty))$. Then
$$
<f_{1}\:,\:\mathcal{H}_{\nu}(f_{2})> \;=\; <\mathcal{H}_{\nu}(f_{1})\:,\:f_{2}>.
$$
\end{proposition}
The proof is immediate using Fubini and the fact that $J_{\nu}(x)$ takes real values
when $\nu $ and $x$ are real.
\begin{corollary} \label{C:isometry}
Assume $\nu$ is real and $\nu>-1$. Let $f_{1},f_{2}\in S_{\nu}([0,\infty))$. Then
$$
<\mathcal{H}_{\nu}(f_{1})\:,\:\mathcal{H}_{\nu}(f_{2})>\; =\; <f_{1}\:,\:f_{2})>.
$$
\end{corollary}

\section{The Induced space for the Discrete Series}
In this section we define an action of the group $G=SL(2,\mathbb{R})$ on
the space $I_{d}$ where $d$ is a positive integer. This is the smooth part
of the full induced space of the Discrete series representation. 
Using the results in the first section we will compute the image of this
space under the inverse of the intertwining map $T_{d}^{-1}$ which was
defined in (\ref{E:Tmu}). The image gives us a ``Kirillov space'' for the full induced.
It is well known that the full induced is reducible and has two 
closed invariant irreducible subspaces. We will later show
one of them is the space $S_{d}([0,\infty))$. In this section we will
show that the image of  $S_{d}$ under the operator $T_{d}$ is invariant.

Let $G=SL(2,\mathbb{R})$. Let $N$ and $A$ be subgroups of $G$ defined by 
$$
N=\left\{n(x)=
\begin{pmatrix}
1 & x \\
0 & 1 \\
\end{pmatrix} :\: x \in \mathbb{R} \right\}
$$
$$
A= \left\{s(z)=
\begin{pmatrix}
z & 0 \\
0 & z^{-1} \\
\end{pmatrix}:\: z\in \mathbb{R}^{*} \right\}.
$$
Let 
$$
w = 
\begin{pmatrix}
0   & -1\\
1   & 0 \\
\end{pmatrix}
$$
and 
$$
r(\theta)=
\begin{pmatrix}
cos(\theta)   & sin(\theta)\\
-sin(\theta)   & cos(\theta) \\
\end{pmatrix}.
$$

\subsection{The asymptotics description of the space $I_{d}$}
Let $d$ be a positive integer. The space $I_{d}$ is 
the space of smooth functions $\phi$ on $\mathbb{R}$ such that
$$
\mathcal{W}_{d}(\phi)(x)=(i)^{d+1}x^{-d-1}\phi(-1/x)
$$
is smooth. This is also the space of smooth functions $\phi$ such that
$$
w(\phi)(x)=x^{-d-1}\phi(-1/x)
$$
is smooth. We first prove some basic properties of this space:
\begin{lemma} \label{L:diff}
The differential operators $d/dx$ and  $D=x(d/dx)$ map $I_{d}$ into itself.
\end{lemma}
\begin{proof}
Let $\phi\in I_{d}$. Then $w(\phi)(x)$ is smooth hence
$w(\phi)'(x)=-(d+1)x^{-d-2}\phi(-1/x)+x^{-d-3}\phi'(-1/x)$ is smooth.
Hence $xw(\phi)'(x)=-(d+1)x^{-d-1}\phi(-1/x)+x^{-d-2}\phi'(-1/x)$ is smooth.
Since $x^{-d-1}\phi(-1/x)$ is smooth it follows that $x^{-d-2}\phi'(-1/x)$ is smooth.
It is clear that $\phi'(x)=d/dx \,\phi$ and $D\phi$ are smooth. To finish the proof 
we need to show that $w(\phi')$ and $w(D\phi)$
are smooth. But $w(D\phi)(x)=-x^{-d-2}\phi'(-1/x)$ which we showed was smooth
and $w(\phi')(x)=-xw(D\phi)(x)$ hence is also smooth.
\end{proof}

We shall now give another description of $I_{d}$. 
\subsubsection{Asymptotic expansions}

We recall some results from the theory of Asymptotic Expansions. For the proofs see 
(\cite{erd}, \cite{van1}, \cite{van2},\cite{olv}).
Let $\phi:\mathbb{R}\to \mathbb{C}$. Let 
$a_{0},a_{1},\ldots$ be complex numbers.
We say that $\phi$ has the asymptotic expansion
$$
\phi(x)\approx \sum_{m=0}^{\infty}a_{m}x^{-m}
$$
at infinity if $\phi(x)-\sum_{m=0}^{N}a_{m}x^{-m}=O(|x|^{-N-1})$ when
$|x|\to \infty$ for all non negative integers $N$ where the implied constant 
is dependent on $N$. (Notice that we have grouped together $\infty$ and 
$-\infty$ which is not needed in a more general definition.)
It is easy to see that the constants $a_{m}$ are determined uniquely by 
$\phi$ although $\phi$ is not determined uniquely by $\{a_{m}\}$. 
(For example, $\phi(x)+e^{-|x|}$ will have the same asymptotic expansion as
$\phi$.). The following results are well known (\cite{van2}):
\begin{lemma}
Let $N$ be a non negative integer and $c$ be a real constant. 
Then $\phi(x)=(x+c)^{-N}$ has an asymptotic expansion of the form
$$
(x+c)^{-N}\approx x^{-N}-Ncx^{-N-1}+...
$$
\end{lemma}
\begin{corollary} \label{C:asympt}
If $\phi(x)$ has an asymptotic expansion then $\phi(x+c)$ has an asymptotic
expansion for every real constant $c$.
\end{corollary}
\begin{proposition}
Assume that $\phi$ is differentiable and both $\phi$ and $\phi'$ have
asymptotic expansions. Then the asymptotic expansion of $\phi'$ is the
derivative term by term of the asymptotic expansion of $\phi$. That is
$$ 
\phi'(x)\approx \sum_{m=0}^{\infty}-ma_{m}x^{-m-1}
$$    
\end{proposition}
\begin{theorem} \label{T:asympt}
Assume that $\phi$ is smooth on $\mathbb{R}$ and let $\gamma(x)=\phi(1/x)$.
Then $\phi$ and all its derivatives have asymptotic expansions on $\mathbb{R}$
if and only if $\gamma$ is smooth at $x=0$. (hence if and only if  
$\gamma$ is smooth on $\mathbb{R}$).
\end{theorem}
\begin{proof}
If $\gamma$ is smooth at $x=0$ then the asymptotic expansions of $\phi$
and all its derivative follow from the Taylor expansion for $\gamma$ and its
derivatives. In the other direction, the smoothness of $\gamma$ at $x=0$ 
follows from the proof of (Theorem 3, \cite{van2}). There it is proved from 
an estimate on $\gamma^{(m)}(x)$. (See \cite{van2} (14) with $l=0$.)
\end{proof}
\begin{proposition}
The space  $I_{d}$ is the space of smooth functions $\phi$ such that $\phi$
and all its derivatives have asymptotic expansions and such that the asymptotic
expansion for $\phi$ is of the form
\begin{equation} \label{E:dasympt}
\phi(x)\approx a_{d+1}x^{-d-1}+\ldots
\end{equation}
\end{proposition}
\begin{proof}
Let $\phi\in I_{d}$ and let $\gamma=w\phi$. Since $\gamma$ is smooth and
since $\phi(x)=x^{-d-1}\gamma(-1/x)$ is smooth, it follows from Theorem~\ref{T:asympt} that 
$\phi$ has the required
asymptotic expansion. Now assume that $\phi$ (and its derivatives)
has the asymptotic exapnsion (\ref{E:dasympt}). Then $\alpha(x)=x^{d+1}\phi(x)$ is smooth and has
an asymptotic exapansion. It follows from Theorem~\ref{T:asympt} that 
$\gamma(x)=\alpha(-1/x)=x^{-d-1}\phi(-1/x)=(w\phi)(x)$ is smooth hence $\phi \in I_{d}$.
\end{proof} 
We now define a representation $\pi_{d}$ of $G$ on  $I_{d}$
in the following way: Let $\phi\in I_{d}$.
\begin{align}
(w\phi)(x)&=(i)^{d+1}\mathcal{W}_{d}(\phi)(x)=x^{-d-1}\phi(-1/x)\\
(n(y)\phi)(x)&=\phi(x+y)\\
(s(z)\phi)(x)&=z^{-d-1}\phi(z^{-2}x)
\end{align}
The fact that $\phi(x+y)\in I_{d}$ follows from Corollary~\ref{C:asympt}.
It is easy to see that these maps give isomorphisms of $I_{d}$. 
More generally we have
\begin{equation} \label{E:SLaction}
\left(\pi_{d}\begin{pmatrix}
p & q \\
r & s \\
\end{pmatrix}\phi \right)(x)=(rx+p)^{-d-1}\phi(\frac{sx+q}{rx+p})
\end{equation}
where the matrix above is in $G$. It is easy to check that this gives a representation
of $G$, that is $\pi_{d}(g_{1}g_{2})=\pi_{k}(g_{1})\pi_{d}(g_{2})$ for all $g_{1},g_{2}\in G$.

We shall now define an ``intertwining operator'' $M_{d}$ on 
$I_{d}$ and compute its
image. This operator is the inverse of the operator $T_{d}$ and will take us back
to the ``Kirillov model''. For $\phi \in I_{d}$ we define
$$
M_{d}(\phi)(y)=|y|^{(-d+1)/2}\hat{\phi}(y)=|y|^{(-d+1)/2}(2\pi)^{-1/2}\int_{-\infty}^{\infty}\phi(x)e^{-ixy}dx.
$$
We view $M_{d}(\phi)(y)$ as a function on $\mathbb{R}-\{0\}$. Its behaviour at zero
is central to the description of this mapping. 
Let $\tilde{S}_{d}=M_{d}(I_{d})$. 
\begin{lemma} \label{L:S}
Let $f\in \tilde{S}_{d}$. Then $f$ is differentiable for every $y\not=0$.
The operator $D=y(d/dy)$ maps $\tilde{S}_{d}$ into itself.
\end{lemma}
\begin{proof}
We first assume that $d>1$. The case $d=1$ will be proved later. Let $\phi \in I_{d}$. Then
$x\phi(x)\in L^{1}(\mathbb{R})$ hence $\hat{\phi}$ is differentiable. 
We have $(D\phi)^{\wedge}(y)=-\hat{\phi}(y)-D\hat{\phi}(y)$
and 
$$
M_{d}(D(\phi))(y)=|y|^{(-d+1)/2}(D\phi)^{\wedge}(y)=
- |y|^{(-d+1)/2}\hat{\phi}(y)-|y|^{(-d+1)/2}D\hat{\phi}(y).
$$
Since $D\phi\in I_{d}$ it follows that $|y|^{(-d+1)/2}D(\hat{\phi})(y)\in \tilde{S}_{d}$.
Now $D(|y|^{(-d+1)/2}\hat{\phi}(y))=\frac{-d+1}{2}|y|^{(-d+1)/2}\hat{\phi}(y)+|y|^{(-d+1)/2}D\hat{\phi}(y)$
Since both summands belong to $\tilde{S}_{d}$ we get our conclusion.
\end{proof}
\begin{corollary}
If $f\in \tilde{S}_{d}$ then $f$ is smooth at every $y\not =0$. 
\end{corollary} 
\begin{proof}
We prove by induction on $m$ that $f$ is $m$ times differentiable. For $m=1$ this
follows from the Lemma. Also by the Lemma $D^{m}f\in \tilde{S}_{d}$ hence it
is differentiable.
But  $D^{m}f(y)$ equals $y^{m}f^{(m)}(y)$ plus a sum involving lower order derivatives
of $f$. Since these summands are all differentiable from the induction assumption,
it follows that $y^{m}f^{(m)}(y)$ is differentiable hence $f^{(m)}(y)$ is differentiable for
$y\not=0$.
\end{proof}
\begin{lemma} \label{L:schwartz}
Let $f\in \tilde{S}_{d}$. Then $f$ is a Schwartz function, that is, $f$ and all its
derivatives are rapidly decreasing at $\infty$ and $-\infty$.
\end{lemma}
\begin{proof}
Let  $\phi \in I_{d}$ be such that $f=M_{d}(\phi)$. Since $\phi$ and all its derivatives are
in $L^{1}(\mathbb{R})$ it follows that $\hat{\phi}$ is rapidly decreasing hence
$f(y)=|y|^{(-d+1)/2}\hat{\phi}(y)$ is rapidly decreasing. Since $D(f)\in \tilde{S}_{d}$
it follows that $D(f)$ is rapidly decreasing hence $f'$ is rapidly decreasing.
Using induction and the fact that $D^{m}(f)$ is rapidly decreasing we get that
$f^{(m)}$ is rapidly decreasing.
\end{proof}
We shall now describe the behaviour at zero of the functions in our space $\tilde{S}_{d}$. We will
also complete the case $d=1$ that was left out in the proof of Lemma~\ref{L:S}.
We let
$$
\phi_{0}(x)=(1+x^{2})^{(-d-1)/2}
$$
and $\phi_{j}(x)=\phi_{0}^{(j)}$, the $j$th derivative of $\phi_{0}$. It is easy to check that 
$\phi_{0}\in I_{d}$ hence by
Lemma~\ref{L:diff},  $\phi_{i}\in I_{d}$ for every positive integer $i$. We shall compute
the functions $f_{j}=(\hat{\phi}_{j})$. From (\cite{grad} 3.771 (2)) it follows that
\begin{align}
f_{0}(y)&=\frac{2^{(d+1)/2}}{\Gamma((d-1)/2)} |y|^{d/2}K_{d/2} (|y|)\\
f_{j}(y)&=(-i)^{j}y^{j}f_{0}(y)=\frac{2^{(d+1)/2}}{\Gamma((d-1)/2)} (-i)^{j}y^{j} |y|^{1/2}K_{d/2} (|y|)
\end{align}
\begin{lemma} \label{L:properties}
$f_{i}(y)$ are smooth on $y\not=0$. They are smooth on the right and smooth
on the left at $y=0$.  $f_{0}(y)$ is continuous at $y=0$. $f_{j}(y)$ is $j$ times
differentiable at $x=0$ and satisfies $f_{j}^{(r)}(0)=0$ for $r=0,1,\ldots,j-1$ and
$f_{j}^{(j)}(0)\not =0$
\end{lemma}
\begin{proof}
The first part of the lemma follows from the fact that $f_{0}(y)$ is a linear conbination
of functions of the form $|y|^{t}e^{-|y|}$ where $t$ is a non-negative integer. This follows
from the fact that $f_{0}$ is even and is given by (\cite{grad} 3.737 (1)).
Since $f_{j}$ is a multiple of $f_{0}$ by a constant times an integer power of $y$ the
smoothness of $f_{j}$ also follows. The continuity of $f_{0}$ follows from the same 
reason or from the fact that it is a Fourier transform of an $L^{1}$ function. Since 
$f_{0}$ is smooth for $y\not=0$ and continuous at $y=0$ it follows that
the function $y^{j}f_{0}(y)$ is $j$ times differentiable and that the $m$th derivative
vanihes at $y=0$ when $m<j$. Since $f_{0}(0)\not=0$ as it is an integral of a positive function, 
it follows that $f_{j}^{(j)}(0)\not= 0$.
\end{proof}
{\bf Proof of Lemma~\ref{L:S} for $d=1$}

We need to prove that
$M_{1}(\phi)(y)$ is differentiable at $y\not=0$. The rest of the proof will follow
the same lines as in the proof of Lemma~\ref{L:S}. Let $\phi \in I_{1}$. Then $\phi$ has
an asymptotic expansion of the form $\phi(x)\approx a_{2}x^{-2}+\ldots$  
Let $\alpha(x)=\phi(x)-a_{2}\phi_{0}(x)=\phi(x)-a_{2}(1+x^2)^{-1}$. 
Then $\alpha(x)$ and $x\alpha(x)$ are in $L^{1}$ hence $\hat{\alpha}$ is differentiable.
Now $\hat{\phi}(y)=\hat{\alpha}(y)+a_{2}\hat{\phi}_{0}(y)=\hat{\alpha}(y)+2a_{2}e^{-|y|}$
is differentiable. The same argument will be used in the proof of the following corollary.
\begin{corollary} \label{C:zero}
Let $f\in \tilde{S}_{d}$.  Then $|y|^{-d+1)/2}f(y)$ is smooth on the right and on the 
left at $y=0$.
\end{corollary}
\begin{proof}
There  exist $\phi\in I_{d}$ such that $|y|^{(d-1)/2}f(y)=\hat{\phi}(y)$ .  
We need to show that $\hat{\phi}$ is smooth on the left and on the right at $y=0$. 
For each integer $m>0$ we can write
$\phi=\alpha_{m}+\sum_{j=0}^{m}a_{j+d-1}\phi_{j}$ so that $\alpha_{m}$ satisfies that 
$x^{r}\alpha_{m}(x)\in L^{1}(\mathbb{R})$
for $r=0,1,\ldots,m+d-1$.  It follows that $\hat{\alpha_{m}}$ is $m+d-1$ times differentiable
on $\mathbb{R}$. But $\hat{\phi}=\hat{\alpha}_{m}+\sum_{j=0}^{m}a_{j+d-1}f_{j}$ and
since $f_{j}$ are smooth on the right and on the left at $y=0$ it follows that
$\hat{\phi}$ is $m+d-1$ smooth on the right and on the left at $y=0$. Since this
is true for all $m$ we get the conclusion.
\end{proof}
\begin{lemma} \label{L:zero}
Let $f\in \tilde{S}_{d}$.  Then $|y|^{(d-1)/2}f(y)$ has continuous $d-1$ derivatives
at $y=0$.
\end{lemma}
\begin{proof}
There  exist $\phi\in I_{d}$ such that $|y|^{d-1/2}f(y)=\hat{\phi}(y)$. From the 
asymptotics of $\phi$ it follows that $\phi(x),x\phi(x),\ldots,x^{d-1}\phi(x)$ are all in 
$L^{1}(\mathbb{R})$. The conclusion follows.
\end{proof}
We will show that Lemma~\ref{L:schwartz}, Corollary~\ref{C:zero} and Lemma~\ref{L:zero} give a complete 
description of $\tilde{S}_{d}$. To do that we define the inverse map $T_{d}$ by $
T_{d}(f)= (|y|^{(d-1)/2}f)^{\lor}$.
That is,
$$
T_{d}(f)(z)=(2\pi)^{-1/2}\int_{-\infty}^{\infty}|y|^{(d-1)/2}f(y)e^{i yz}dy
$$
Notice that if $f$ vanishes on $(-\infty,0)$ then this is the same as $T_{d}$
that was defined in (\ref{E:Tmu}).
\begin{proposition} \label{P:in}
Let $f$ be such that $|y|^{(-d-1)/2}f(y)$ is smooth on $[0,\infty)$ and on $(-\infty,0]$
(but not necessarily smooth at $y=0$)
and such that $f$ and all its derivatives are rapidly decreasing. Then 
$f\in \tilde{S}_{d}$.
\end{proposition}
\begin{proof}
It is enough to show that  $T_{d}(f)\in I_{d}$ since $f=M_{d}(T_{d}(f))$ hence is
in $\tilde{S}_{d}$. To do that we write $f=f_{+}+f_{-}$ where $f_{+}$ vanishes on
$(0,-\infty)$ and is smooth on $[0,\infty)$ and $f_{-}$ vanishes on $(0,\infty)$
and is smooth on $(-\infty,0]$. (Notice that $f(0)=0$. ) Now 
$f_{+}\in S_{d}([0,\infty))$ and by Proposition~\ref{P:map} $T_{d}(f_{+})\in I_{d}$.
Similar arguments as in the proof of  Proposition~\ref{P:map}  will show that 
$T_{d}(f_{-})\in I_{d}$ hence $T_{d}(f)\in I_{d}$.
\end{proof}
\begin{theorem} \label{T:desc}
$\tilde{S}_{d}$ is the set of smooth functions $f$ on $\mathbb{R}-\{0\}$ such
that $f$ and all its derivatives are rapidly decreasing at $\pm\infty $ and such that
$|y|^{(d-1)/2}f(y)$ is smooth on the right and on the left
at $y=0$ and has $d-1$ continuous derivatives at $y=0$. 
(When $d=1$, this means that $|y|^{d-1}f(y)=f(y)$ is continuous.)
\end{theorem}
\begin{proof}
By Lemma~\ref{L:schwartz}, Corollary~\ref{C:zero} and Lemma~\ref{L:zero} $\tilde{S}_{d}$ is contained 
in the space of functions satisfying the above conditions. Now assume that $f$ is as 
above and we will show that it is in $\tilde{S}_{d}$, that is, it is the image under 
$M_{d}$ of a function in $I_{d}$.  Let $\tilde{f}= |y|^{(d-1)/2}f(y)$. 
By the properties of $\{f_{k}\}$ in Lemma~\ref{L:properties} we can find using
a triangulation argument constants $c_{0},\ldots, c_{d-1}$ so that
$$
\tilde{h}=\tilde{f}-\sum_{j=0}^{d-1}c_{j}\tilde{f}_{j}
$$
vanishes at zero and all its first $d-1$ derivatives vanish at zero. It follows that
$h(y)=|y|^{(-d+1)/2}\tilde{h}$ satisfies the conditions of Proposition~\ref{P:in} hence is
in $\tilde{S}_{d}$. Hence $f=h+\sum_{j=0}^{d-1}c_{j}f_{j}$ is in $\tilde{S}_{d}$. 
\end{proof}
Since $I_{d}$ is isomorphic to $\tilde{S}_{d}$ under the isomorphism $M_{d}$
it follows that the action of $G$ on $I_{d}$ induces an action of $G$ on 
$\tilde{S}_{d}$.
Let $I_d^{+}$ be the subspace of $I_{d}$ which is given by 
\begin{equation} \label{E:Idplus}
I_d^{+}=T_{d}(S_{d}([0,\infty)).
\end{equation} 
By Corollary~\ref{C:stab} $\mathcal{H}_{d}$ preserves the space
$S_{d}([0,\infty))$. Hence by Theorem~\ref{T:desc}, $w$ preserves the space $I_d^{+}$.
It is easy to see (see (\ref{E:action}) that the induced action of $n(y)$ and $s(z)$ on 
$\tilde{S}_{d}$ stabilize the space $S_{d}([0,\infty))$. Hence it follows that $n(y)$ and $s(z)$
stabilize the space $I_d^{+}$.Since $w$, $n(y)$
and $s(z)$  generate $G$ (when $y$ and $z$ vary)  we get the following corollary:
\begin{corollary} \label{C:inv}
$I_d^{+}$ is a $G$ invariant subspace of $I_{d}$.
\end{corollary}
Let
\begin{equation} \label{E:minus}
S_{d}((-\infty,0])=
\{f:(-\infty,0],\to \mathbb{C}\; |\; f(x)=x^{(d+1)/2}f_{1}(x) 
\;\text{and} \;f_{1}\in S((-\infty,0])\}
\end{equation}
For $y\in (-\infty,0]$ and $f\in S_{d}((-\infty,0])$ we define
$$
\mathcal{H}_{d}(f)(y)=
\int_{-\infty}^{0}f(x)\sqrt{|xy|}J_{d}(2\sqrt{|xy|})\frac{dx}{|x|}.
$$
It is clear that $S_{d}((-\infty,0])$ is a subspace of $\tilde{S}_{d}$.  
It follows from Corollary~\ref{C:stab} that 
$\mathcal{H}_{d}$ stabilizes $S_{d}((-\infty,0])$ and that the induced action of 
$n(y)$ and $s(z)$ given by
\begin{equation} \label{L:minusaction}
(n(y)f)(x)=e^{iyx}f(x),\;\;\;(s(z)f)(x)=f(z^{2}x)
\end{equation}
also stabilize
$S_{d}((-\infty,0])$.
Define
\begin{equation} \label{E:Idminus}
I_d^{-}=T_{d}(S_{d}((-\infty,0])).
\end{equation}
Then $I_d^{-}$ is a $G$ invariant subspace of $I_{d}$. We will show in Section~\ref{S:frechet} that $I_{d}^{\pm}$ is
a closed subspace under the Frechet topology on $I_{d}$ hence (from the theory of $(\mathfrak{g},K)$ modules)
$I_{d}^{+}\oplus I_{d}^{-}$ is the unique
maximal closed invariant subspace of $I_{d}$.

\section{The Kirillov Model of the discrete series}

The Kirillov model is a particular realization of the representations of
$SL(2,\mathbb{R})$ or $GL(2,\mathbb{R})$ (or more generally, representations
of $GL(n,F)$ where $F$ is a local field) with a prescribed action of the
Borel subgroup. In this section we will describe the Kirillov model of the
discrete series representations of $G=SL(2,\mathbb{R})$. Our main theorem
of this section and this paper is a description of the smooth space of the Kirillov model.
This Theorem will be proved in the next section. 

\subsection{The action of the group in the Kirillov model}
Let $d$ be a positive integer. Using the map $T_{d}$ and its inverse
$M_{d}$ we can move the action of $G$ from the space $I_{d}$ to the space
$S_{d}$. We shall denote this action by $R_d^{+}$. That is, 
if $g\in G$ then the 
action of $g$ on $f \in S_{d}$ is given
by 
$$
R_{d}^{+}(g)(f)=M_{d}(gT_{d}(f)).
$$
Thus we obtain the following formulas. For each  
$f\in S_{d}([0,\infty))$ we have
\begin{align} \label{E:Lieaction}
&(R_{d}^{+}(n(y))f)(x)=e^{iyx}f(x)\\
&(R_d^{+}(s(z))f)(x)=f(z^{2}x)\\
&(R_{d}^{+}(w)f)(x)=i^{-(d+1)}\mathcal{H}_{d}(f)(x)
\end{align}
Let 
$$
X=\begin{pmatrix}
0 & 1 \\
0 & 0 \\
\end{pmatrix},\;\;\;
H=\begin{pmatrix}
1 & 0 \\
0 & -1 \\
\end{pmatrix},\;\;\;
Y=\begin{pmatrix}
0 & 0 \\
1 & 0 \\
\end{pmatrix}
$$
be a basis for the Lie Algebra. The action of the Lie Algebra on the 
Kirillov model above is given by
\begin{align} \label{E:action}
&(R_{d}^{+}(X)f)(x)=ixf(x)\\
&(R_d^{+}(H)f)(x)=2xf'(x)\\
&(R_{d}^{+}(Y)f)(x)=-i\frac{d^2-1}{4x}f(x)+ixf''(x)
\end{align}
Another (non-isomorphic) action can be obtained by considering the
action of $G$ on $S_{d}((-\infty,0])$ (see (\ref{E:minus}) and the remark below it) 
and moving this action
to $S_{d}([0,\infty))$. The action is given by
\begin{align}
&(R_{d}^{-}(n(y))f)(x)=e^{-iyx}f(x)\\
&(R_d^{-}(s(z))f)(x)=f(z^{2}x)\\
&(R_{d}^{-}(w)f)(x)=i^{d+1}\mathcal{H}_{d}(f)(x)
\end{align}
It is easy to see that the action of $n(y)$ and $s(z)$ extend to 
the space $L^{2}\left( (0,\infty),dx/x \right)$. It follows from Corollary~\ref{C:isometry}
that the action
of $w$ also extends to $L^{2}\left((0,\infty),dx/x \right)$. We denote again by 
$R_{d}^{\pm}$ the representation on the space 
$\mathcal{H}_{d}=L^{2}\left((0,\infty),dx/x\right)$ obtained in such a way.
This is called the Kirillov Hilbert representation of $G$. 
\begin{proposition}
$R_{d}^{\pm}$ is a strongly continous unitary representation on 
$L^{2}\left((0,\infty),dx/x \right)$. 
\end{proposition}
\begin{proof}
It is enough to show that the map $g\to R_{d}^{\pm}(g)f$ is continuous
at $g=w$ for every $f\in L^{2}\left((0,\infty),dx/x\right)$.(\cite{kna} p.11 ) This is easy 
to show for a characteristic function of an interval using the formulas in 
(\ref{E:action}) and by approximation for a general function.
\end{proof}
\begin{proposition}
$R_{d}^{\pm}$ is an irreducible representation.
\end{proposition}
\begin{proof}
This representation is already reducible when restricted to the Borel subgroup.
(See \cite{kna} proof of Proposition 2.6)
\end{proof}
Our main theorem of this section is the following.
\begin{theorem} \label{T:smooth}
The space of smooth vectors of $\mathcal{H}_{d}=L^{2}\left((0,\infty),dx/x\right)$
under the action of $R_{d}^{\pm}$ is $S_{d}\left([0,\infty)\right)$. That is, 
$\mathcal{H}_{d}^{\infty}=S_{d}\left([0,\infty)\right)$.
\end{theorem}
To prove Theorem~\ref{T:smooth} we will need to compare the  Fr\'{e}chet
topologies on the induced space and on the Kirillov space. 

\section{The Fr\'{e}chet topology}\label{S:frechet}
In this section we describe the  Fr\'{e}chet topology on the different
smooth spaces that we are considering. This topology will play a role
in determining the smooth space of the Kirillov model.

Let $G=SL(2,\mathbb{R})$. Let $d$ be a positive integer and let
\begin{equation} \label{E:induced}
Ind_{d}=\{F:G\to \mathbb{C}| \;\text{$F$ is smooth and 
$\;F(n(x)s(a)g)=a^{d+1}F(g), \;$ for all 
$x\in \mathbb{R},a\in \mathbb{R}^{*}$}\}
\end{equation}
$G$ acts on this space by right translations and $\mathfrak{g}=\text{Lie}(G)$
acts by left invariant differential operators and induces an action
of the enveloping algebra $U(\mathfrak{g})$ on $Ind_{d}$. 
There is an $L^{2}$ norm defined
on this space by
\begin{equation} \label{E:norm}
||F||^{2}=\frac{1}{2\pi}\int_{0}^{2\pi}|F(r(\theta))|^{2}d\theta
\end{equation}
The space $Ind_{d}$ is given the Fr\'{e}chet topology defined by the seminorms
$$
||F||_{2,D}=||D(F)||, \;\;\;D\in  U(\mathfrak{g}).
$$
We will now show that the space $Ind_{d}$ is isomorphic as a vector space to the space $I_{d}$ 
defined in (\ref{E:Inu}).
For each $F\in Ind_{d}$ define
\begin{equation} \label{E:intertwiner}
\phi_{F}(x)=F(wn(x))\: .
\end{equation}
The following proposition is well known:
\begin{proposition}
The mapping $F\to \phi_{F}$ is an isomorphism of vector
spaces. 
\end{proposition}
\begin{proof}
Let $F\in Ind_{d}$. Since $F$ is smooth on $G$ it follows that $\phi_{F}$ is
smooth on $\mathbb{R}$. It is easy to check that $\phi_{wF}=w\phi_{F}$
where $w\phi_{F}(x)=x^{-d-1}\phi_{F}(-1/x)$ hence $w\phi_{F}$ is smooth and
we have that the mapping sends $Ind_{d}$ into $I_{d}$. Since the set of 
elements of the form $n(y)s(a)wn(x)$ is dense in $G$, it follows that a 
function $F$ in $Ind_{d}$ is determined by its restriction to $wn(x)$ hence
the mapping $F\to \phi_{F}$ is one to one. It remains to show that this 
mapping is onto. We have
\begin{equation} \label{E:theta}
wn(x)=s((1+x^2)^{-1/2})n(-x)r(\theta_{x})
\end{equation}
where $\theta_{x}\in (0,\pi)$ is determined by the equalities
$sin(\theta_{x})=(1+x^{2})^{-1/2},\;cos(\theta_{x})=-x(1+x^{2})^{-1/2}$.
Given $\phi \in I_{d}$ we define
\begin{equation} \label{E:fdefine}
F_{\phi}(r(\theta_{x}))=(1+x^{2})^{(d+1)/2}\phi(x)
\end{equation}
for $\theta_{x}\in (0,\pi)$ and define
$$
F_{\phi}(r(\pi))=F_{\phi}(-I)=
\lim_{x\to \infty}(1+x^{2})^{(d+1)/2}\phi(x)=(w\phi)(0).
$$
$F_{\phi}$ is defined on all of $G$ via the equivariance property
\begin{equation}\label{E:equi}
F_{\phi}(n(y)s(a)r(\theta))=sgn(a)a^{d+1}F_{\phi}(r(\theta))
\end{equation}
and the fact that each $g\in G$ can be written uniquely in the form
$g=n(y)s(a)r(\theta)$ for $\theta \in (0,\pi]$.
Since $F_{\phi}(r(\theta-\pi))=F_{\phi}(r(-\pi)r(\theta))=
(-1)^{d+1} F_{\phi}(r(\theta))$ it follows that $F_{\phi}$ satisfies
(\ref{E:equi}) for all $\theta$. It remains to show that 
$F_{\phi}(r(\theta))$ is smooth as a function of $\theta$ (which will show
that $F_{\phi}$ is smooth on $G$.) This is obvious by (\ref{E:fdefine})
when $\theta \in (0,\pi)$. When $\theta = \pi$ we notice that 
$F_{\phi}(r(\theta))=F_{w\phi}(r(\theta-\pi/2))$ hence the smoothness 
of $F_{\phi}$ at $\theta=\pi$ is equivalent to the smoothness of $F_{w\phi}$
at $\theta=\pi/2$ which was already established. It is also immediate that
$\phi_{F_{\phi}}=\phi$.
\end{proof}
The isomorphism $F\to \phi_{F}$ defined above induces  
a Fr\'{e}chet topology on the space $I_{d}$. To give this topology explicetely we need:
\begin{proposition}
Let $F\in Ind_{d}$ and $\phi_{F}:\mathbb{R}\to \mathbb{C}$ as defined above.
Then
$$
||F||^{2}=\frac{1}{\pi}\int_{-\infty}^{\infty}(1+x^{2})^{d}|\phi_{F}(x)|^{2}dx
$$
\end{proposition}
\begin{proof}
Using (\ref{E:theta}) we have
\begin{equation}
\begin{split}
\frac{1}{\pi}\int_{-\infty}^{\infty}(1+x^{2})^{d}|\phi_{F}(x)|^{2}dx &=
\frac{1}{\pi}\int_{-\infty}^{\infty}(1+x^{2})^{d}(1+x^{2})^{-d-1}
|F(r(\theta_{x}))|^{2}dx\\
&=\frac{1}{\pi}\int_{-\infty}^{\infty}(1+x^{2})^{-1}
|F(r(cot^{-1}(-x))|^{2}dx \\
&=\frac{1}{\pi}\int_{0}^{\pi}|F(r(\theta))|^{2}d\theta\\
&=\frac{1}{2\pi}\int_{0}^{2\pi}|F(r(\theta))|^{2}d\theta\\
&=||F||^{2}
\end{split}
\end{equation}
\end{proof}
Hence we define a Fr\'{e}chet topology on $I_d$ using the seminorms
$$
||\phi||_{d,D}=\int_{-\infty}^{\infty}(1+x^{2})^{d}|(D\phi)(x)|^{2}dx, \;\;\;D\in U(\mathfrak{g}).
$$
It is useful to replace the seminorms given in the  Fr\'{e}chet topology of $Ind_{d}$ with an equivalent
set of seminorms which gives the topology of uniform convergence of functions
and derivatives. That is, For $F \in Ind_{d}$ define
$||F||_{\infty}=\text{max}_{\theta}|F(\theta)|$ be the L-infinity norm on 
$Ind_{d}$.
It is well known (\cite{wal1} Theorem 1.8) that the set of seminorms 
$$
||F||_{\infty,D}=||D(F)||_{\infty}, D\in U(\mathfrak{g})
$$
is an equivalent set of seminorms to the set above and gives the same Fr\'{e}chet
topology. Moreover, let 
$$
\partial = 
\begin{pmatrix}
0   & -1\\
1   & 0 \\
\end{pmatrix},\;\text{and}\;\;\;\partial(F)(r(\theta))=\frac{\partial}{\partial \theta}F((r(\theta))).
$$
Then it is enough to have the seminorms
$$
||F||_{\infty,n}=||\partial^{n}(F)||_{\infty}.
$$
Hence we get the following Corollary:
\begin{corollary}
The Fr\'{e}chet topology on $I_{d}$ is given by the set of seminorms
$$
||f||_{\infty,d,D}=|(1+x^{2})^{(d+1)/2}(D\phi)(x)|_{\infty},\;\;\; 
D\in U(\mathfrak{g})
$$
or by the subset of seminorms
$$
||f||_{\infty,d,\partial^{n}}=|(1+x^{2})^{(d+1)/2}
(\partial\phi)(x)|_{\infty},\;\;\; n=0,1,...
$$
\end{corollary}
\begin{proposition} \label{P:cont}
The functional $l_{\lambda,n}$ on $I_d$ defined by
$$
l_{\lambda,n}(\phi)=\int_{-\infty}^{\infty}t^{n}\phi(t)e^{-i\lambda t}dt
$$
is continous for $n=0,....,d-1$ and $\lambda \in \mathbb{R}$.
\end{proposition}
\begin{proof}
Assume that $\phi_{m} \in I_{d}$ satisfy $\lim_{m\to \infty}\phi_{m}=0$. in the  Fr\'{e}chet topology.
Then $l_{\infty}((1+t^{2})^{(d+1)/2}\phi)\to 0$ hence there exist
a sequence of positive constants $c_{m}$ such that 
$$
|\phi_{m}(t)|\leq c_{m}(1+t^2)^{-(d+1)/2},\;\;\;\
\text{for all $t\in \mathbb{R}$.}
$$
and $\lim_{m\to \infty} c_{m}=0$. It follows that
$$
|l_{\lambda,n}(\phi_{m})|\leq c_{m}\int_{-\infty}^{\infty}|t|^{n}
(1+t^2)^{-(d+1)/2}dt.
$$
The integral above converges by the assumption on $n$ hence we have that
$\lim_{m\to \infty}l_{\lambda,n}(\phi_{m})=0$.
\end{proof}
Let $I_{d}^{+}$ and $I_{d}^{-}$ be the subspaces of $I_{d}$ defined in 
(\ref{E:Idplus},(\ref{E:Idminus}).
Then it is easy to see that
$$
I_{d}^{+}=\{\phi\in I_{d} \;|\; l_{\lambda,n}(\phi)=0,\;\;
\text{for all $0\leq n\leq d-1$ and $\lambda \leq 0$} \}\\
$$
and
$$
I_{d}^{-}=\{\phi\in I_{d} \;|\; l_{\lambda,n}(\phi)=0,\;\;
\text{for all $0\leq n\leq d-1$ and $\lambda \geq 0$} \}
$$
\begin{corollary} \label{C:closed}
$I_{d}^{+}$ and $I_{d}^{-}$ are closed invariant subspaces of $I_{d}$.
They are irreducible as Fr\'{e}chet representations.
\end{corollary}
\begin{proof}
By Proposition~\ref{P:cont} they are closed as intersection of closed
subspaces. By Corollary~\ref{C:inv} and the discussion below it they are $G$ invariant subspaces. 
By the theory of $(\mathfrak{g},K)$ modules (See for example \cite{god} p. 2.8 Theorem 2) $I_d$ has exactly 
three nontrivial invariant
subspaces, two irreducible (lowest weight and highest weight) subspaces and
their sum. It follows that $I_{d}^{+}$ and $I_{d}^{-}$ are irreducible.
\end{proof}
Our next goal is to show that $S_{d}([0,\infty))$ is a space of smooth vectors
in $\mathbf{H}_{d}$. That is, our first step in showing that $\mathbf{H}_{d}^{\infty}=S_{d}([0,\infty))$
is to show that $S_{d}([0,\infty))\subseteq  \mathbf{H}_{d}$. To do that we will need to compare the $L^{2}$
inner product on  $S_{d}([0,\infty))$ with the standard (compact) inner product on $Ind_{d}$. Combining the 
isomorphism $T_d$ from $S_{d}$ to $I_{d}$ and the isomorphism $\phi \to F_{\phi}$ defined in 
(\ref{E:fdefine}),(\ref{E:equi})  from $I_{d}$
to $Ind_{d}$ we get a one to one mapping from  $S_{d}([0,\infty))$ to $Ind_{d}$ and a subspace of $Ind_{d}$. 
This infinite
dimensional subspace has two different inner products. The standard inner product on the induced space 
(\ref{E:norm})
and the inner product induced from the $L^{2}$ inner product on $S_{d}$. Following (\cite{god} (288)) we will 
give an explicit formula for this induced inner product and compare between the two.

We shall define an invariant ``norm'' on $Ind_{d}$. (It is in fact  a norm on the
invariant subspaces)
$$
||f||_{I}^{2}=\frac{1}{2\pi}\int_{0}^{2\pi}A(f)(r(\theta))\overline{f(r(\theta))}
d\theta
$$
where $A(f)$ is the intertwining operator defined by
$$
A(f)(g)=\int_{-\infty}^{\infty}f(w^{-1}n(x)g)dx,\;\;\;g\in G,f\in Ind_{d}.
$$
\begin{remark}
Since $f(r(\theta-\pi))=(-1)^{d}f(r(\theta))$ for every $\theta$ and
since the same holds for $A(f)$ we have that
$$
||f||_{I}=\frac{1}{\pi}\int_{0}^{\pi}A(f)(r(\theta))\overline{f(r(\theta))}
d\theta
$$
\end{remark}
For $x\in \mathbb{R}$ we define $\theta_{x}\in (0,\pi)$ by
$$
\theta_{x}=\text{cot}^{-1}(-x)
$$
We have
\begin{equation} \label{E:theta1}
wn(x)=s((1+x^2)^{-1/2})n(-x)r(\theta_{x})
\end{equation}
and
\begin{equation} \label{E:theta2}
w^{-1}n(x)=s(-(1+x^2)^{-1/2})n(-x)r(\theta_{x})
\end{equation}
Hence using the change of variable $\theta=\text{cot}^{-1}(x)$ we have
$$
A(f)(g)=\int_{0}^{\pi}f(r(\theta)g)(sin(\theta))^{d-1}d\theta
$$
and in particular
$|A(f)(r(\theta))|\leq \pi \, L^{\infty}(f)$.
Hence we have proved:
\begin{lemma}
$$
||f||_{I}\leq  \pi \, L^{\infty}(f)
$$
\end{lemma}
\begin{proposition}
Let $f\in I_{d}^{+}$. Then
\begin{equation} \label{E:inner}
\frac{(-2\pi i)^{d}}{2((d-1)!)}\int_{-\infty}^{\infty}\int_{-\infty}^{\infty} 
sgn(y)y^{d-1}f(x-y)\overline{f(x)}dydx =||M(f)||_{2}^{2}
\end{equation}
where $||M(f)||_{2}$ is the $L^{2}$ norm of $M(f)$ in the space $L^{2}\left( (0,\infty),dx/x) \right)$.
\end{proposition}
\begin{proof}
The proof is the same as in (\cite{god} p. 1.62 - 1.64). We will use the Tate identity
$$
\int_{-\infty}^{\infty}f(y)sgn(y)y^{d-1}dy=
\frac{2((d-1)!)}{(-2\pi i)^{d}}\int_{-\infty}^{\infty}\hat{f}(x)x^{-d}dx
$$
which is valid for $f\in I_{d}$. It follows from this identity that the left
hand side of (\ref{E:inner}) is equal to
\begin{equation} \notag
\begin{split}
&\int_{-\infty}^{\infty} \overline{f(x)}e^{2\pi i xy}
\int_{-\infty}^{\infty}\hat{f}(y)y^{-d}dydx\\
&= \int_{-\infty}^{\infty}\overline{\hat{f}(y)}\hat{f}(y)y^{-d}dy\\
&=\int_{-\infty}^{\infty}\overline{\hat{f}(y)y^{(-d+1)/2}}\hat{f}(y)y^{-(d+1)/2}dy/y\\
&=||M(f)||_{2}
\end{split}
\end{equation}
\end{proof}
For $f\in S_{d}$ we attach $F\in Ind_{d}$ by  $F \,= \, F_{T_{d}(f)}$ where
$F_{\phi}$ is defined in (\ref{E:fdefine}),(\ref{E:equi}).
\begin{proposition}
Let $f\in S_{d}$. Then
$$
||f||_{2}=\frac{2((d-1)!)}{(-2\pi i)^{d}} ||F||_{I}
$$
\end{proposition}
\begin{proof}
\begin{equation} \notag
\begin{split} \notag
||f||_{I}^{2}&=\frac{1}{\pi}\int_{0}^{\pi}A(f)(r(\theta))\overline{f(r(\theta))}
d\theta\\
&=\frac{1}{\pi}\int_{0}^{\pi}
A(f)(s(\text{sin}(\theta))n(\text{cot}(\theta))r(\theta)))
\overline{f(s(\text{sin}(\theta))n(\text{cot}(\theta))r(\theta))}d\theta\\
&=\int_{-\infty}^{\infty}A(f)\overline{f(wn(x))}dx\\
&=\int_{-\infty}^{\infty}\int_{-\infty}^{\infty}f(w^{-1}n(y)wn(x)\overline{f(wn(x))}
dydx
\end{split}
\end{equation}
\end{proof}
\begin{corollary} \label{C:norm}
Let $f$ and $F$ be as above then there exist a constant $c$ (depending on $d$ but not on $f$)
such that $||f||_{2} \leq c L^{\infty}(F)$.
\end{corollary}
\begin{corollary}
$S_{d}$ is a space of smooth functions in the Hilbert representations space 
$\mathbf{H}_{d}$ of $R_{d}$.
\end{corollary}
\begin{proof}
By a (\cite{wal1}, Theorem 1.8) the smooth vectors in the Hilbert representation associated with $Ind_{d}$ are the 
smooth functions on $G$ satisfying 
(\ref{E:induced}). (That is, the space $Ind_{d}$ is the space of smooth vectors in the appropriate $L^{2}$ space)
Hence, the map $f$ to $F$ given above is a $G$ invariant map sending the space $S_{d}([0,\infty))$ to a space
of smooth vectors in $Ind_{d}$. Using the Corollary~\ref{C:norm} and the definition of smooth vectors we get our result. 
\end{proof} 
%

\section{$(\mathfrak{g},K)$ modules and Fr\'{e}chet  spaces
isomorphism}
In this section we will show that $S_{d}([0,\infty))$ is the smooth space of the representation
$R_{d}^{\pm}$ on the space $\mathbf{H}_{d}=L^{2}((0,\infty))$. To do that we will show that the operator 
$M_{d}$ from $I_{d}^{+}$ to $S_{d}([0,\infty))$ is an isomorphism of the
$(\mathfrak{g},K)$ modules of $K$-finite vectors and also an isomorphism between the spaces of 
smooth vectors. 

\subsection{$K$-finite vectors}
Using the map $M_{d}$ we can find $K$ finite vectors in $\mathcal{H}_{d}=
L^{2}((0,\infty),dy/y)$. 
\begin{lemma} \label{L:Kfinite}
Every function of the form $y^{(d+1)/2}p(y)e^{-y}$ where $p(y)$ is a polynomial
is a $K$-finite vector in $\mathbf{H}_{d}$.
\end{lemma}
\begin{proof}
We let $F_{0}\in Ind_{d}$ be defined by $F_{0}(r(\theta))=e^{-(d+1)\theta}$ and extended
to $G$ as in (\ref{E:equi}). Then $\phi(x)=\phi_{F_{0}}(x)=(1+x^2)^{-(d+1)/2}e^{i(d+1)tan^{-1}(x)}$
and by (\cite{grad} 3.944 (5),(6)) we have that $M_{d}(\phi)= \frac{\sqrt{2\pi}}{d!}y^{(d+1)/2}e^{-y}$. 
Now the n-th derivative  
$\phi(x)$ is also a $K$-finite vector since it is the application
$n$ times of the differential operator $X\in \mathfrak{g}$. Since
$M_{d}(\phi^{(n)})=\lambda y^{n}y^{(d+1)/2}e^{-y}$ for a nonzero
constant $\lambda$ it follows that $p(y)y^{(d+1)/2}e^{-y}$  is a $K$ finite vector
in $\mathcal{H}_{d}$ for every polynomial $p(y)$ and we have proved
the Lemma.
\end{proof}
The Leguerre orthogonal polynomials $L_{n}^{d}(x)$ are defined by the formula
$$
L_{n}^{d}(x)=e^{x}\frac{x^{-d}}{n!}\frac{d^{n}}{dx^{n}}(e^{-x}x^{n+d})
$$
It is well known (see \cite{leb} (4.21.1) and 4.23 Theorem 3) that for a fixed integer
$d\geq 0$, the set of functions $\phi_n(x)=\left(\frac{n!}{(n+d)!}\right)^{1/2}x^{d/2}L_{n}^{d}(x)$, $n=0,1,...$
is a complete orthonormal system for $L^{2}((0,\infty),dx)$. Hence we get
\begin{proposition}
Let $d$ be a positive integer. The set of functions 
$e_{n}(x)=\left(\frac{n!}{(n+d)!2^{d+1}}\right)^{1/2}x^{(d+1)/2}e^{-x}\;\; n=0,1,...$ is a
complete orthonormal
system of $K$ finite vectors (in fact, $K$ eigenfunctions) for $\mathcal{H}_{d}=L^{2}((0,\infty),dx/x)$.
\end{proposition}
\begin{proof}
By Lemma~\ref{L:Kfinite} the functions $e_{n}(x)$ are all $K$-finite. By the remark on the Laguerre polynomials
it follows that the set of functions $x^{-1/2}e_{n}(x),\;\;n=0,1,...$ is a complete orthonormal set for
$L^{2}((0,\infty),dx)$. If $f(x)\in L^{2}((0,\infty),dx/x)$ is orthogonal to all functions $e_{n}(x),\;\;n=0,1,...$
then $x^{-1/2}f(x)\in L^{2}((0,\infty),dx)$ is orthogonal to $x^{-1/2}e_{n}(x),\;\;n=0,1,...$ hence is the zero 
function.
\end{proof}
\begin{corollary}
The set of all $K$ finite vectors in $\mathcal{H}_{d}$ is the set of functions
$x^{(d+1)/2}e^{-x}p(x)$ where $p(x)$ is a polynomial. 
\end{corollary}
\begin{proof}
Otherwise we would be able to find a $K$ finite eigenfunction which is 
orthogonal to all the $e_{n}$ which is a contradiction.
\end{proof}
\begin{corollary} \label{C:iso}
The $(\mathfrak{g},K)$ module of $K$-finite vectors in $I_{d}^{+}$ is isomorphic
to the $(\mathfrak{g},K)$ module of $K$-finite vectors in $\mathcal{H}_{d}$
and both are irrecducible.
\end{corollary}
\begin{proof}
The operator $T_{d}$ is a one to one intertwining operator between the 
$(\mathfrak{g},K)$ of the $K$-finite vectors in $I_{d}^{+}$ and the 
$(\mathfrak{g},K)$ module of the $K$-finite vectors in $\mathbf{H}_{d}$. By the 
above lemma it is onto.
\end{proof}
Our main result of this paper is the following:
\begin{theorem} \label{T:main1}
The space $S_{d}([0,\infty))$ is the space of smooth vectors in  $\mathcal{H}_{d}$, 
that is, $S_{d}([0,\infty))=\mathcal{H}_{d}^{\infty}$
\end{theorem}
\begin{proof}
By Corollary~\ref{C:norm} the operator $M_{d}$ is a smooth intertwining operator between $I_{d}^{+}$
and $\mathcal{H}_{d}^{\infty}$ whose image is $S_{d}([0,\infty))$. By Corollary~\ref{C:iso}
$M_{d}$ restricts to a
$(\mathfrak{g},K)$ isomorphism of the spaces of $K$-finite vectors. By
a theorem of Casselman and Wallach (\cite{wal2}, Theorem 11.6.7) there is a 
unique continuous extension of
such isomrphism to an isomorphism of the smooth spaces of each representation.
It follows that $M_{d}$ is that extension and that $M_{d}$ is onto 
$\mathcal{H}_{d}^{\infty}$ hence  $S_{d}([0,\infty))=\mathcal{H}_{d}^{\infty}$.
\end{proof} 

\section{The Kirillov model for the discrete series representations of
$GL(2,\mathbb{R})$} 

In this section we use are previous results to describe the Kirillov model
and in particular the smooth space of the Kirillov model of the discrete
series representations of $GL(2,\mathbb{R})$.

The discrete series representations of $GL(2,\mathbb{R})$ are parametrized by 
two real characters $\chi_{1}(t)=|t|^{s_{1}}\text{sgn}(t)^{m_{1}}$,  
$\chi_{2}(t)=|t|^{s_{2}}\text{sgn}(t)^{m_{2}}$ where 
$s_{1},s_{2} \in \mathbb{C}$ and $m_{1},m_{2}\in \{0,1\}$ are such
that 
\begin{equation} \label{E:parameter}
|t|^{s_{1}-s_{2}}\text{sgn}(t)^{m_{1}-m_{2}}=t^{d}\text{sgn}(t)
\end{equation}
for some positive integer $d$. The smooth space of the discrete series is
a subspace of the induced representation $\text{Ind}(\chi_{1},\chi_{2})$
which is given by the space of smooth functions 
$F:GL(2,\mathbb{R})\rightarrow \mathbb{C}$ satisfying
$$
F\left(\begin{pmatrix}
a & x \\
0 & b \\
\end{pmatrix}g \right)=\chi_{1}(a)\chi_{2}(b)|a/b|^{1/2}F(g)
$$
for every $g\in GL(2,\mathbb{R})$. The action of $GL(2,\mathbb{R})$
on this space is by right translations. We denote the central character of this 
representation by $\omega$ where
$$
\omega(b)=|b|^{s_{1}+s_{2}+1}\text{sgn}(b)^{m_{1}+m_{2}}.
$$
It is easy to see that the mapping
$F\to \phi_{F}$ in (\ref{E:intertwiner}) gives an isomorphism between $\text{Ind}(\chi_{1},\chi_{2})$ 
and $I_{d}$
where $d$ is the positive integer given by (\ref{E:parameter}). Moreover, the
induced action of the subgroup $SL(2,\mathbb{R})$ on  $I_{d}$ is given by
(\ref{E:SLaction}).
It follows from Corollary~\ref{C:closed} and the theory of $(\mathfrak{g},K)$
modules (\cite{god} p. 2.7, 2.8) that the space $I_{d}^{+}\oplus I_{d}^{-}$
is an irreducible closed subspace of $I_{d}$ under the action of 
$GL(2,\mathbb{R})$. To obtain the smooth space of the Kirillov model we
define a mapping $M_{\chi_{1},\chi_{2}}$ from $I_{d}$ by
$$
M_{\chi_{1},\chi_{2}}(\phi)(y)=|y|^{(d+1)/2}\text{sgn}(y)^{m_2}\hat{\phi}(y)
$$
\begin{remark}
The mapping $F \to M_{\chi_{1},\chi_{2}}(\phi_{F})(y)$ from the subspace of
the induced representation $\text{Ind}(\chi_{1},\chi_{2})$ is identical to the
map 
$$
F\to W_{F}
\begin{pmatrix}
|y|^{1/2}\text{sgn}(y) & 0 \\
0 & |y|^{-1/2} \\
\end{pmatrix}
$$
defined in (\cite{god}, (75)) where $W_{F}$ is the Whittaker function 
associated with $F$. This space of functions can be called a ``normalized'' 
Kirillov model. It is related to the ``standard'' Kirillov model,
$$
F\to W_{F}
\begin{pmatrix}
y & 0 \\
0 & 1 \\
\end{pmatrix}
$$
via the relation in (\cite{god}, (75)) which is:
$$
 W_{F}
\begin{pmatrix}
|y|^{1/2}\text{sgn}(y) & 0 \\
0 & |y|^{-1/2} \\
\end{pmatrix}=
\omega(|y|^{1/2})
W_{F}
\begin{pmatrix}
y & 0 \\
0 & 1 \\
\end{pmatrix}
$$
\end{remark} 
It is easy to see that 
$M_{\chi_{1},\chi_{2}}(\phi)(y)=
\text{sgn}(y)^{m_{2}}M_{d}(\phi)(y)$.
Since  $M_{d}$ maps the space $I_{d}^{+}$ onto $S_{d}([0,\infty))$
and the sapce  $I_{d}^{-}$ onto $S_{d}((-\infty,0])$ it follows that
$M_{\chi_{1},\chi_{2}}$ does the same. Hence
$M_{\chi_{1},\chi_{2}}$ sends the space $I_{d}^{+}\oplus I_{d}^{-}$ onto the 
space
$\mathcal{K}_{d}$
which can be described as follows: $\mathcal{K}_{d}$ is the space of
smooth functions $f:(\mathbb{R}-\{0\})\rightarrow \mathbb{C}$ such that
$f$ and all its derivatives are rapidly decreasing at $\pm \infty$ and such
that the function $g(x)=|x|^{(d+1)/2}f(x)$ is smooth on the right and left
at $x=0$.
An example for such a function is the function $|x|^{(d+1)/2}e^{-|x|}$.

The proof for these assertions is the same as the proof of 
Theorem~\ref{T:main}. The idea
is to define an irreducible representation of $GL(2,\mathbb{R})$ on an
$L^{2}$ space and to show using the various Frechet topologies that the
smooth space of this representation is $\mathcal{K}_{d}$. We now 
describe the Hilbert space and the action of $GL(2,\mathbb{R})$ on this
space. The details of the proofs are left to the reader.

We define $\mathcal{V}_{d}=L^{2}(\mathbb{R},dx/|x|)$. 
We define a representation $R_{\chi_{1},\chi_{2}}$ on 
$\mathcal{V}_{d}$
by 
\begin{align} \notag
&\left(R_{\chi_{1},\chi_{2}}
\begin{pmatrix}
1 & y \\
0 & 1 \\
\end{pmatrix} 
f\right)(x)=e^{iyx}f(x) \notag\\
&\left(R_{\chi_{1},\chi_{2}}
\begin{pmatrix}
a & 0 \\
0 & 1 \\
\end{pmatrix} 
f\right)(x)=|a|^{(s_{1}+s_{2}+1)/2}f(ax)=\omega(|a|^{1/2})f(ax)\notag\\
&\left(R_{\chi_{1},\chi_{2}}
\begin{pmatrix}
b & 0 \\
0 & b \\
\end{pmatrix}
f\right)(x)=|b|^{s_{1}+s_{2}+1}\text{sgn}(b)^{m_{1}+m_{2}}f(x)=\omega(b)f(x)\notag\\
&\left(R_{\chi_{1},\chi_{2}}
\begin{pmatrix}
0 & -1 \\
1 & 0 \\
\end{pmatrix}
f\right)(y)=\text{sgn}(y)^{m_2+d+1} 
\int_{-\infty}^{\infty}f(x)j_{d}(xy)dx/|x| \\
\end{align}
where
\begin{equation}
j_{d}(x)=
  \begin{cases}
(i)^{-(d+1)}\sqrt{x}J_{d}(2\sqrt{x}) & \text{if $x>0$} \\
0 & \text{if $x<0$} \\
   \end{cases}
\end{equation}
and the integral is defined as an $L^{2}$ extension of a unitary operator
on the space $\mathcal{K}(s_1)$.
\begin{theorem}
The representation $R_{\chi_{1},\chi_{2}}$ on the space
$L^{2}(\mathbb{R},dx/|x|)$ is strongly continuous and irreducible.
It is unitary if the central character $\omega=\chi_{1}\chi_{2}$ is unitary.
\end{theorem}
\begin{theorem}
The smooth space of the representation  $R_{\chi_{1},\chi_{2}}$ on 
$L^{2}(\mathbb{R},dx/|x|)$ is the space $\mathcal{K}_{\chi_{1},\chi_{2}}$.
\end{theorem}

\bibliographystyle{amsplain}
\bibliography{big}
\end{document}